\documentclass[runningheads,a4paper]{llncs} 
\usepackage{url}
\usepackage{bm}
\usepackage{enumerate}
\usepackage{amsmath,amssymb,amsbsy,latexsym}
\usepackage{graphicx}
\usepackage{xcolor}
\newcommand\ul\underline
\newcommand\noul{}

\newcommand{\dd}{{\rm d}}

\newcommand\ddh[1]{\frac{\dd^{#1}}{\dd h^{#1}}}

\newcommand\nL{\mathcal{L}}

\newcommand\nS{\mathcal{S}}

\newcommand\Order{{\mathcal{O}}}

\newcommand\NN{\mathbb N}

\begin{document}
%
%
\mainmatter              
\pagestyle{myheadings}   
\title{Setup of
       Order Conditions for Splitting Methods}
\titlerunning{Order Conditions for Splitting Methods}  
%
\author{Winfried Auzinger \inst{1} \and Wolfgang Herfort \inst{1}
\and Harald Hofst{\"a}tter \inst{1} \and Othmar~Koch \inst{2}}
\authorrunning{Winfried Auzinger et al.} 
\institute{Technische Universit{\"a}t Wien, Austria \\
\email{w.auzinger@tuwien.ac.at,
       w.herfort@tuwien.ac.at,
       hofi@harald-hofstaetter.at}, \\
\url{www.asc.tuwien.ac.at/~winfried},
\url{www.asc.tuwien.ac.at/~herfort},
\url{www.harald-hofstaetter.at}
\and
Universit{\"a}t Wien, Austria \\
\email{othmar@othmar-koch.org}, \\
\url{www.othmar-koch.org}
}

\maketitle 

\begin{abstract} 
This article is based on~\cite{auzingeretal13c} and~\cite{auzingeretal15K1},
where an approach based on Taylor expansion and the structure of
its leading term as an element of a free Lie algebra was described
for the setup of a system of order conditions for operator splitting methods.
Along with a brief review of these materials and some theoretical background,
we discuss the implementation of the ideas from these papers
in computer algebra, in particular
using\footnote{Maple is a trademark of $ \text{MapleSoft}^{\text{TM}} $.}
Maple~18.
A parallel version of such a code is described.
\end{abstract}
\section{Introduction} \label{sec:intro}
The construction of higher order discretization schemes of one-step type
for the numerical solution of evolution equations is typically based
on the setup and solution of a large system of polynomial equations
for a number of unknown coefficients. Classical examples are Runge-Kutta
methods, and their various modifications, see e.g.~\cite{ketchetal13}.

To design particular schemes, we need to understand
\begin{itemize}
\item[(i)] how to generate a system of algebraic equations for the coefficients sought,
\smallskip
\item[(ii)] how to solve the resulting system of polynomial equations.
\end{itemize}
Here, we focus on (i) which depends on the particular class of
methods one is interested in. We consider {\em splitting methods,}
which are based on the
idea of approximating the exact flow of an evolution equation by
compositions based on (usually two) separate subflows
which are easier to evaluate.

Computer algebra is an indispensable tool for solving such a problem,
and there are different algorithmic approaches.
In general there is a tradeoff between `manual' a priori analysis
and machine driven automatization.

For splitting methods,
a well-known approach is based on recursive application of the
Baker-Campbell-Hausdorff (BCH) formula, see~\cite{haireretal06}.
Instead, we follow another approach based on Taylor expansion and
a theoretical result concerning the structure of the leading term
in this expansion.
This has the advantage that explicit knowledge of the BCH-coefficients
is not required. Moreover, our approach adapts easily to splitting
into more than two parts, and even to pairs of splitting schemes akin
to Runge-Kutta methods.

Topic (ii) is not discussed in this paper.
Details concerning the theoretical background and a discussion concerning
concrete results and optimized schemes obtained
are given in~\cite{auzingeretal15K1},
and a collection of optimized schemes can be found at~\cite{splithp}.
We note that a related approach has recently
also been considered in~\cite{blanesetal13a}.

\subsection{Splitting methods for the integration of evolution equations}
\label{subsec:splitting}
In many applications, the right hand side $ F(u) $ of an evolution equation
\begin{equation} \label{AB-problem}
\partial_t u(t) = F(u(t)) = A(u(t)) + B(u(t), \quad t \geq 0,
\quad u(0)~\text{given,}
\end{equation}
splits up in a natural way into two terms $ A(u) $ and $ B(u) $,
where the separate integration of the subproblems
\begin{equation*}
\partial_t u(t) = A(u(t)), \qquad \partial_t u(t) = B(u(t))
\end{equation*}
is much easier to accomplish than for the original problem.
\begin{example} \label{exa:lie-trotter}
The solution of a linear ODE system with constant coefficients,
\begin{equation*}
\partial_t u(t) = (A+B)\,u(t),
\end{equation*}
is given by
\begin{equation*}
u(t) = e^{t(A+B)}\,u(0).
\end{equation*}
The simplest splitting approximation (`Lie-Trotter'),
starting at some initial value $ u $ and applied with a time step of length $ t=h $,
is given by
\begin{equation*}
\nS(h,u) = e^{h B}\,e^{h A}\,u \approx e^{h(A+B)}u.
\end{equation*}
This is not exact (unless $ AB=BA $), but it satisfies
\begin{equation*}
\|(e^{h B}\,e^{h A} - e^{h(A+B)})u\| = \Order(h^2)
\quad \text{for}~~ h \to 0,
\end{equation*}
and the error of this approximation depends on behavior of the commutator
$ [A,B] = AB-BA $.
\qed
\end{example}
A general splitting method takes steps of the
form\footnote{By $ \Phi_F $ we denote the flow
              associated with the equation $ \partial_t u = F(u) $,
              and $ \Phi_A,\,\Phi_B $ are defined analogously.}
\begin{subequations} \label{AB-scheme}
\begin{equation} \label{AB-scheme-1}
\nS(h,u) = \nS_s(h,\nS_{s-1}(h,\ldots,\nS_1(h,u))) \approx \phi_F(h,u),
\end{equation}
with
\begin{equation} \label{AB-scheme-2}
\nS_j(h,v) = \phi_B(b_j\,h,\phi_A(a_j\,h,v)),
\end{equation}
\end{subequations}
where the (real or complex) coefficients $ a_j,b_j $ have to be found
such that a certain desired order of approximation for $ h \to 0 $ is obtained.

The local error of a splitting step is denoted by
\begin{equation} \label{local_error_notation}
\nS(h,u) - \phi_F(h,u) =: \nL(h,u).
\end{equation}

For our present purpose of finding asymptotic order conditions it is sufficient
to consider the case of a linear system, $ F(u) = F\,u = A\,u + B\,u $
with linear operators $ A $ and~$ B $.
We denote
\begin{equation*}
A_j = a_j\,A,~ B_j = b_j\,B, \quad j=1 \ldots s.
\end{equation*}
Then,
\begin{subequations} \label{AB-scheme-linear}
\begin{equation}  \label{AB-scheme-linear-1}
\nS(h,u) = \nS(h) u, \quad
\nS(h) = \nS_s(h)\,\nS_{s-1}(h)\,\cdots\,\nS_1(h) \approx e^{h F},
\end{equation}
with
\begin{equation}  \label{AB-scheme-linear-2}
\nS_j(h) = e^{h B_j}\,e^{h A_j}, \quad j=1 \ldots s.
\end{equation}
\end{subequations}
For the linear case the local error~\eqref{local_error_notation}
is of the form $ \nL(h) u $ with the linear operator
$ \nL(h) = \nS(h) - e^{hF} $.
\subsection{Commutators} \label{subsec:lie}
Commutators of the involved operators play a central role.
For formal consistency, we call $ A $ and $ B $ the `commutators of degree 1'.
There is (up to sign)
one non-vanishing\footnote{`Non-vanishing' means non-vanishing in general
                            (generic case, with no special assumptions on
                            $ A $ and $ B $).}
commutator of degree~2,
\begin{equation*}
[A,B] := A\,B - B\,A,
\end{equation*}
and there are two non-vanishing commutators of degree~3,
\begin{equation*}
[A,[A,B]] = A\,[A,B] - [A,B]\,A,
\quad
[[A,B],B] = [A,B]\,B - B\,[A,B],
\end{equation*}
and so on; see Section~\ref{subsec:locerr} for commutators of higher degrees.
\section{Taylor expansion of the local error} \label{sec:taylor}
\subsection{Representation of Taylor coefficients} \label{subsec:taycoe}
Consider the Taylor expansion, about $ h=0 $, of the local error operator $ \nL(h) $
of a consistent one-step method (satisfying the basic consistency
condition $ \nL(0)=0 $),
\begin{equation} \label{lerr-lead}
\nL(h) =
\sum_{q=1}^{p} \frac{h^q}{q!}\,\ddh{q}\,\nL(h)\,\Big|_{h=0}
+\, \Order(h^{p+1}).
\end{equation}
The method is of asymptotic order $ p $ iff $ \nL(h) = \Order(h^{p+1}) $
for $ h \to 0 $; thus the conditions for order $ \geq p $ are given by
\begin{equation} \label{OC-Tay}
\ddh{}\,\nL(h)\,\Big|_{h=0} = \;\ldots\; = \ddh{p}\,\nL(h)\,\Big|_{h=0} =\, 0.
\end{equation}
The formulas in~\eqref{OC-Tay} need to be presented in a more
explicit form, involving the operators $A$ and $B$.
For a splitting method~\eqref{AB-scheme-linear},
a calculation based on the Leibniz formula for higher derivatives
shows\footnote{If $ A $ and $ B $ commute, i.e., $ AB=BA $, then all
               these expressions vanish.}
(see~\cite{auzingeretal15K1})
\begin{equation} \label{OC-Tay-1-AB}
\ddh{q}\,\nL(h)\,\Big|_{h=0}
= \sum_{|{\bm k}|=q} \dbinom{q}{{\bm k}}
  \prod\limits_{j=s \ldots 1}\;
  \sum_{\ell=0}^{k_j} \dbinom{k_j}{\ell}\,B_j^{\ell}\,A_j^{k_j-\ell}
  \;-\; (A+B)^q,
\end{equation}
with $ {\bm k}=(k_1,\ldots,k_s) \in \NN_0^s $.

\paragraph{Representation of~\eqref{OC-Tay-1-AB} in Maple.}
The non-commuting operators $ A $ and $ B $ are represented by
symbolic variables {\tt A} and {\tt B}, which can be declared to
be non-commutative making use of the corresponding feature
implemented in the package {\tt Physics}.
Now it is straightforward to generate the sum~\eqref{OC-Tay-1-AB},
with unspecified coefficients $ a_j,b_j $,
using standard combinatorial tools;
for details see Section~\ref{sec:impl}.
\subsection{The leading term of the local error expansion} \label{subsec:locerr}
Formally, the multinomial sums in the expressions~\eqref{OC-Tay-1-AB}
are multivariate homogeneous polynomials of total degree $ q $ in the variables
$ a_j,b_j,\,j=1 \ldots s $, and the coefficients of these polynomials are
power products of total degree $ q $ composed of powers of the
non-commutative symbols~$ A $ and~$ B $.
\begin{example}[\,{\rm\cite{auzingeretal15K1}}] \label{exa:s=p=2}
For $ s=2 $ we obtain
\begin{align*}
\ddh{}\,\nL(h)\,\Big|_{h=0} &= \noul{(a_1+a_2)}\,A + \noul{(b_1+b_2)}\,B
                 \;-\;(A+B), \\[2\jot]
\ddh{2}\,\nL(h)\,\Big|_{h=0} &= ((a_1+a_2)^2)\,A^2  \label {OC-s=2-2}
                   + \noul{(2\,a_2\,b_1)}\,A\,B \\
                & \quad {}+ (2\,a_1\,b_1 + 2\,a_1\,b_2 + 2\,a_2\,b_2)\,B\,A
                   + ((b_1+b_2)^2)\,B^2 \\[\jot]
                 & \quad \,-\,(A^2 + A\,B + B\,A +B^2).
\end{align*}
The consistency condition for order $ p \geq 1 $ reads
$ \ddh{}\,\nL(h)\,\big|_{h=0} = 0 $, which is
equivalent to $ a_1+a_2=1 $ and $ b_1+b_2=1 $.

At first sight, for order $ p \geq 2 $ we need 4, or (at second sight)
2~additional equations to be satisfied, such that
$ \ddh{2}\,\nL(h)\,\big|_{h=0}=0 $.
However, assuming that the conditions for order $ p \geq 1 $ are satisfied,
the second derivative $ \ddh{2}\,\nL(h)\,\big|_{h=0} $
simplifies to the commutator expression
\begin{equation*}
\ddh{2}\,\nL(h)\,\Big|_{h=0} = \noul{(2\,a_2\,b_1 - 1)}\,[A,B],
\end{equation*}
giving the single additional condition $ 2\,a_2\,b_1 = 1$ for order $ p \geq 2 $.
Assuming now that $ a_1,a_2 $ and $ b_1,b_2 $ are chosen such
that all conditions for $ p \geq 2 $ are satisfied,
the third derivative $ \ddh{3}\,\nL(h)\,\big|_{h=0} $,
which now represents the leading term of the local error,
simplifies to a linear combination of the commutators
$ [A,[A,B]] $ and $ [[A,B],B] $, of degree~3, namely
\begin{equation*}
\qquad\quad
\ddh{3}\,\nL(h)\,\Big|_{h=0} = \noul{(3\,a_2^2\,b_1 - 1)}\,[A,[A,B]]
                               + \noul{(3\,a_2\,b_1^2 - 1)}\,[[A,B],B].
\qquad\qed
\end{equation*}
\end{example}
\begin{remark}
The classical second-order Strang splitting method corresponds to the choice
$ a_1=\frac{1}{2},\,b_1=1,\,a_2=\frac{1}{2},\,b_2=0 $, or
$ a_1=0,\,b_1=\frac{1}{2},\,a_2=1,\,b_2=\frac{1}{2} $.
\end{remark}
The observation from this simple example generalizes as follows:
\begin{proposition} \label{pro:leading-lie}
The leading term $ \ddh{p+1}\,\nL(h)\,\big|_{h=0} $
of the Taylor expansion of the local error $ \nL(h) $
of a splitting method of order $ p $ is a Lie element, i.e.,
it is a linear combination of commutators of degree $ p+1 $.
\end{proposition}
\begin{proof}
See~\cite{auzingeretal13c,haireretal06}.
\qed
\end{proof}
\begin{example} \label{exa:higher-comm}
Assume that the coefficients $ a_j,b_j,\,j=1 \ldots s $ have been found
such that the associated splitting scheme is of order $ p \geq 3 $
(this necessitates $ s \geq 3 $). This means that
\begin{equation*}
\ddh{}\,\nL(h)\,\Big|_{h=0} =\, \ddh{2}\,\nL(h)\,\Big|_{h=0}
                            =\, \ddh{3}\,\nL(h)\,\Big|_{h=0} =\, 0,
\end{equation*}
and from Proposition~\ref{pro:leading-lie} we know that
\begin{equation*}
\ddh{4}\,\nL(h)\,\Big|_{h=0} =
\gamma_1\,[A,[A,[A,B]]] + \gamma_2\,[A,[[A,B],B]] + \gamma_3\,[[[[A,B],B],B]
\end{equation*}
holds, with certain coefficients $ \gamma_k $
depending on the $ a_j $ and $ b_j $.
Here we have made use of the fact that there are three independent commutators
of degree $4$ in $A$ and $B$.
\qed
\end{example}
Targeting for higher-order methods one needs to know a {\em basis of commutators}\,
up to a certain degree. The answer to this question is
known, and a full set of independent commutators of degree $ q $ can be represented
by a set of words of length $ q $ over the alphabet $ \{ A,B \} $.
A prominent example is the set of {\em Lyndon-Shirshov words}\,
(see~\cite{bokutetal2006}) displayed in Table~\ref{tab:Lyndon-AB}.
A combinatorial algorithm due to Duval~\cite{duval88} can be used to
generate this table.

Here, for instance, the word {\tt AABBB} represents the commutator
\begin{align*}
& [A,[[[A,B],B],B]] = \\
& \quad A^2 B^3 - 3ABAB^2 + 3AB^2AB - 2AB^3A + 3BAB^2A - 3B^2ABA + B^3A^2,
\end{align*}
with leading power product $ A^2 B^3 = AABBB $ (w.r.t.\ lexicographical order).
\begin{table}[!ht]
\begin{center}
\begin{small}
\begin{tabular}{|r|r|l|}
\hline
\,$q$\,&$ L_q $  &\;Lyndon-Shirshov words over the alphabet $ \{\mathtt{A,B}\} $ $\vphantom{\sum_A^A}$  \\ \hline
  1\,    &   2\, & \;$ {\mathtt{A}},\,{\mathtt{B}} $          $\vphantom{\sum_A^{A^A}}$  \\
  2\,    &   1\, & \;$ {\mathtt{AB}} $              $\vphantom{\sum_A^A}$  \\
  3\,    &   2\, & \;$ {\mathtt{AAB}},\,{\mathtt{ABB}} $  $\vphantom{\sum_A^A}$  \\
  4\,    &   3\, & \;$ {\mathtt{AAAB}},\,{\mathtt{AABB}},\,{\mathtt{ABBB}} $  $\vphantom{\sum_A^A}$  \\
  5\,    &   6\, & \;$ {\mathtt{AAAAB}},\,{\mathtt{AAABB}},\,{\mathtt{AABAB}},\,
                       {\mathtt{AABBB}},\,{\mathtt{ABABB}},\,{\mathtt{ABBBB}} $  $\vphantom{\sum_A^A}$  \\
  6\,    &   9\, & \;$ {\mathtt{AAAAAB}},\,{\mathtt{AAAABB}},\,{\mathtt{AAABAB}},\,
                       {\mathtt{AAABBB}},\,{\mathtt{AABABB}},\,
                       {\mathtt{AABBAB}},\,{\mathtt{AABBBB}},\,{\mathtt{ABABBB}},\,{\mathtt{ABBBBB}} $ $\vphantom{\sum_A^A}$ \\
  7\,    &   18\,& ~\,\ldots  $\vphantom{\sum_A^A}$  \\
  8\,    &   30\,& ~\,\ldots  $\vphantom{\sum_A^A}$  \\
  9\,    &   56\,& ~\,\ldots  $\vphantom{\sum_A^A}$  \\
 10\,    & \,99\,& ~\,\ldots  $\vphantom{\sum^A}$  \\[-\jot]
\vdots~ & \,\vdots~ & ~\,$\ddots$  $\vphantom{\sum_A}$  \\[\jot] \hline
\end{tabular}
\newline
\end{small}
\caption{$ L_q $ is the number of words of length~$ q $. \label{tab:Lyndon-AB}}
\end{center}
\end{table}
\subsection{The algorithm: implicit recursive elimination} \label{subsec:ire}
On the basis of Proposition~\ref{pro:leading-lie},
and with a table of Lyndon-Shirshov words available,
we can build up a set of conditions for order $ \geq p $
for a splitting method with $ s $ stages in the following way
(recall the notation $ A_j := a_j\,A $, $ B_j = b_j\,B $):

{\em
\medskip\noindent
      For $ q=1 \ldots p $\,:
      \begin{itemize}
      \item
      Generate the symbolic expressions~\eqref{OC-Tay-1-AB}
      in the indeterminate coefficients $ a_j,b_j $ and
      the non-commutative variables {\tt A} and {\tt B}.
      \medskip
      \item
      Extract the coefficients of the power products (of degree $ q $)
      represented by all Lyndon-Shirshov words of length $ q $,
      resulting in a set of $ L_q $ polynomials
      $ P_{q,k}(a_j,b_j) $ of degree $ q $
      in the coefficients $ a_j $ and $ b_j $.
      \end{itemize}
}
\noindent
The resulting set of $ \sum_{q=1}^{p} L_q $ multivariate polynomial equations
\begin{equation} \label{OCSYS}
P_{q,k}(a_j,b_j) = 0, \quad k=1 \ldots L_q, ~~ q=1 \ldots p
\end{equation}
represents the desired conditions for order $ p $.

We call this procedure {\em implicit recursive elimination,} because the
equations generated in this way are correct in an `a~posteriori' sense
(cf.\ Example~\ref{exa:s=p=2}):
\begin{subequations}
\begin{enumerate}[--~]
\item For $ q=1 $, the basic consistency equations
      \begin{equation} \label{OC1}
      \begin{aligned}
      P_{1,1}(a_j,b_j) &= a_1 + \ldots + a_s - 1 = 0, \\
      P_{1,2}(a_j,b_j) &= b_1 + \ldots + b_s - 1 = 0,
      \end{aligned}
      \end{equation}
      are obtained.
\item {\em Assume}\, that~\eqref{OC1} is satisfied. Then, due to
      Proposition~\ref{pro:leading-lie}, the additional (quadratic) equation
      (note that $ L_2=1 $)
      \begin{equation} \label{OC2}
      P_{2,1}(a_j,b_j) = 0,
      \end{equation}
      represents one additional condition for a scheme of order $ p=2 $.
\item {\em Assume}\, that~\eqref{OC1} and~\eqref{OC2} are satisfied.
      Then, due to Proposition~\ref{pro:leading-lie},
      the additional (cubic) equations (note that $ L_3=2 $)
      \begin{equation} \label{OC3}
      P_{3,1}(a_j,b_j) = P_{3,2}(a_j,b_j) = 0,
      \end{equation}
      represent two additional conditions for a scheme of order $ p=3 $.
\item The process is continued in the same manner.
\end{enumerate}
\end{subequations}
If we (later) find a solution $ \{ a_j,b_j,\, j=1 \ldots s \} $
of the resulting system
\begin{equation*}
\eqref{OCSYS}=\{ \eqref{OC1},\eqref{OC2},\eqref{OC3},\ldots \}
\end{equation*}
of multivariate polynomial equations, this means that
\begin{align*}
\text{\eqref{OC1} is satisfied}
&~\Rightarrow~ \text{condition \eqref{OC2} is correct,} \\
\text{\eqref{OC2} is also satisfied}
&~\Rightarrow~ \text{condition \eqref{OC3} is correct,}
\end{align*}
and so on. By induction we conclude that the whole procedure is correct.
See~\cite{auzingeretal15K1}
for a more detailed exposition of this argument.
\begin{remark}
In addition, it makes sense to generate the additional conditions for
order $ p+1 $. Even if we do not solve for these conditions,
they represent the leading term of the local error, and this can be used
to search for optimized solutions for order $ p $, where the coefficients in
$ \ddh{p+1}\,\nL(h)\,\big|_{h=0} $ become minimal in size.
\end{remark}
\section{A parallel implementation} \label{sec:impl}

In our Maple code, a table of Lyndon-Shirshov words up to a fixed length
(corresponding to the maximal order aimed for; see Table~\ref{tab:Lyndon-AB})
is included as static data. The procedure {\tt Order\_conditions} displayed
below generates a set of order conditions using the algorithm described
in Section~\ref{subsec:ire}.
\begin{itemize}
\item Fist of all, we activate the package {\tt Physics}
      and declare the symbols {\tt A} and {\tt B} as non-commutative.
\item For organizing the multinomial expansion according to~\eqref{OC-Tay-1-AB}
      we use standard functions from the packages {\tt combinat}
      and {\tt combstruct}.
\item The number of terms during each stage rapidly increases as
      more stages are to be computed. Therefore we have implemented a parallel
      version based on the package {\tt Grid}. Parallelization is taken into account as follows:
      \begin{itemize}
      \item On a multi-core processor\footnote{E.g., on an Intel i7 processor,
                                               6~cores are available.
                                               The hyper-threading feature
                                               enables the use of 12 parallel threads.},
            all threads execute the same code.
            Each thread identifies itself via a call to {\tt MyNode()},
            and this is used to control execution.
            Communication between the threads is realized via message passing.
      \item Thread~0 is the master thread controlling the overall execution.
      \item For $ q=1 \ldots p $\,:
            \begin{itemize}
            \item
            Each of the working threads generates
            symbolic expressions of the form (recall $ A_j=a_j\,A $,
            $ B_j=b_j\,B $)
            \begin{equation*}
            \Pi_{\bm k} := \dbinom{q}{{\bm k}}
            \prod\limits_{j=s \ldots 1}\;
            \sum_{\ell=0}^{k_j} \dbinom{k_j}{\ell}\,B_j^{\ell}\,A_j^{k_j-\ell},
            \quad {\bm k} \in \NN_0^s,
            \end{equation*}
            appearing in the sum~\eqref{OC-Tay-1-AB}. Here the work is
            equidistributed over the threads, i.e., each of them generates
            a subset of $ \{ \Pi_{\bm k},~ {\bm k} \in \NN_0^s \} $ in parallel.
            \item
            For each of these expressions $ \Pi_{\bm k} $,
            the coefficients of all Lyndon-Shirshov monomials of degree $ q $
            are computed, and the according subsets of coefficients
            are summed up in parallel.
            \item
            Finally, the master thread 0 sums up the results
            received from all the working threads. This results in the set
            of multivariate polynomials representing the order conditions
            at level~$ q $.
            \end{itemize}
      \end{itemize}
\item The Maple code displayed below is, to some extent,
      to be read as pseudo-code.
      For simplicity of presentation we have ignored some technicalities,
      e.g., concerning the proper indexing of combinatorial tupels, etc.
      The original, working code is available from the authors.
\end{itemize}
\begin{small}
\begin{verbatim}
> with(combinat)
> with(combstruct)
> with(Grid)
> with(Physics)
>  Setup(noncommutativeprefix={A,B})

> Order_conditions := proc()
  global p,s,OC,    # I/O parameters via global variables
         Lyndon     # assume that table of Lyndon monomials is available
  this_thread := MyNode()   # each thread identifies itself
  max_threads := NumNodes() # number of available threads
  for j from 1 to s do
      A_j[j] := a[j]*A
      B_j[j] := b[j]*B
      term[-1][j] := 1
  end do
  OC=[0$p]
  for q from 1 to p do
    if this_thread>0 then  # working threads start computing
                           # master thread 0 is waiting
       Mn := allstructs(Composition(q+2),size=2)
       for j from 1 to s do
           term[q-1][j] := 0
           for mn from 1 to nops(Mn) do
               term[q-1][j] :=
                 term[q-1][j] +
                   multinomial(q,Mn[mn])*B_j[j]^Mn[mn][2]*A_j[j]^Mn[mn][1]
           end do
       end do
       k := iterstructs(Composition(q+s),size=s)
       OC_q_this_thread := [0$nops(Lyndon[q])]
       while not finished(k) do # generate expansion (7) term by term
         Ms := nextstruct(k)
         if get_active(this_thread) then # get_active:
                                         #   auxiliary Boolean function
                                         #   for equidistributing workload
            Pi_k := 1
            for j from s to 1 by -1 do
                Pi_k := Pi_k*term[Ms[j]-1][j]
            end do
            Pi_k := multinomial(q,Ms)*expand(Pi_k)
            OC_q_this_thread := # compare coefficients of Lyndon monomials
              OC_q_this_thread +
                [seq(coeff(Pi_k,Lyndon[q][l]),l=1..nops(Lyndon[q]))]
         end if
       end do
       Send(0,OC_q_this_thread) # send partial sum to master thread
    else # master thread 0 receives and sums up
         #   partial results from working threads
       OC[q] := [(-1)$nops(Lyndon[q])] # initialize sum
       for i_thread from 1 to max_threads-1 do
           OC[q] := OC[q] + Receive(i_thread)
       end do
    end if
end do
end proc

> # Example:
> p := 4
> s := 4
> Launch(Order_conditions,imports=["p","s"],exports=["OC"])  # run

> OC[1]
  [a[1]+a[2]+a[3]+a[4]-1,
   b[1]+b[2]+b[3]+b[4]-1]

> OC[2]
  [2*a[2]*b[1]+2*a[3]*b[1]+2*a[3]*b[2]
   +2*a[4]*b[1]+2*a[4]*b[2]+2*a[4]*b[3]-1]

> OC[3]
  [3*a[2]^2*b[1]+6*a[2]*a[3]*b[1]+6*a[2]*a[4]*b[1]
   +3*a[3]^2*b[1]+3*a[3]^2*b[2]+6*a[3]*a[4]*b[1]+6*a[3]*a[4]*b[2]
    +3*a[4]^2*b[1]+3*a[4]^2*b[2]+3*a[4]^2*b[3]-1,
   3*a[2]*b[1]^2+3*a[3]*b[1]^2+6*a[3]*b[1]*b[2]
    +3*a[3]*b[2]^2+3*a[4]*b[1]^2+6*a[4]*b[1]*b[2]+6*a[4]*b[1]*b[3]
     +3*a[4]*b[2]^2+6*a[4]*b[2]*b[3]+3*a[4]*b[3]^2-1]

> OC[4]
  [4*a[2]^3*b[1]+12*a[2]^2*a[3]*b[1]+12*a[2]^2*a[4]*b[1]
   +12*a[2]*a[3]^2*b[1]+24*a[2]*a[3]*a[4]*b[1]+12*a[2]*a[4]^2*b[1]
    +4*a[3]^3*b[1]+4*a[3]^3*b[2]+12*a[3]^2*a[4]*b[1]
     +12*a[3]^2*a[4]*b[2]+12*a[3]*a[4]^2*b[1]+12*a[3]*a[4]^2*b[2]
      +4*a[4]^3*b[1]+4*a[4]^3*b[2]+4*a[4]^3*b[3]-1,
   6*a[2]^2*b[1]^2+12*a[2]*a[3]*b[1]^2+12*a[2]*a[4]*b[1]^2
    +6*a[3]^2*b[1]^2+12*a[3]^2*b[1]*b[2]+6*a[3]^2*b[2]^2
     +12*a[3]*a[4]*b[1]^2+24*a[3]*a[4]*b[1]*b[2]+12*a[3]*a[4]*b[2]^2
      +6*a[4]^2*b[1]^2+12*a[4]^2*b[1]*b[2]+12*a[4]^2*b[1]*b[3]
       +6*a[4]^2*b[2]^2+12*a[4]^2*b[2]*b[3]+6*a[4]^2*b[3]^2-1,
   4*a[2]*b[1]^3+4*a[3]*b[1]^3+12*a[3]*b[1]^2*b[2]
    +12*a[3]*b[1]*b[2]^2+4*a[3]*b[2]^3+4*a[4]*b[1]^3
     +12*a[4]*b[1]^2*b[2]+12*a[4]*b[1]^2*b[3]+12*a[4]*b[1]*b[2]^2
      +24*a[4]*b[1]*b[2]*b[3]+12*a[4]*b[1]*b[3]^2+4*a[4]*b[2]^3
       +12*a[4]*b[2]^2*b[3]+12*a[4]*b[2]*b[3]^2+4*a[4]*b[3]^3-1]
\end{verbatim}
\end{small}

\medskip
For practical use some further tools have been developed, e.g.\
for generating tables of polynomial coefficients for further use,
e.g., by numerical software other than Maple.
This latter job can also be parallelized.
\subsection{Special cases} \label{subsec:special}
Some special cases are of interest:
\begin{itemize}
\item {\em Symmetric schemes}\, are characterized by the property
      $ \nS(-h,\nS(h,u)) = u $. Here, either $ a_1=0 $ or $ b_s=0 $,
      and the remaining coefficient sets $ (a_j) $ and $ (b_j) $ are palindromic.
      Symmetric schemes have an even order $ p $, and the order conditions
      for even orders need not be included; see~\cite{haireretal06}.
      Thus, we use a special ansatz and generate a reduced set of
      equations.
\smallskip
\item {\em Palindromic schemes}\, were introduced in~\cite{auzingeretal15K1}
      and characterized by the property $ \nS(-h,{\check\nS}(h,u)) = u $,
      where $ \check\nS $ denotes the scheme $ \nS $ with the role of $ A $
      and $ B $ interchanged.
      In this case, the full coefficient set
      \begin{equation*}
      (a_1,b_1,\ldots,a_s,b_s)
      \end{equation*}
      is palindromic. As for symmetric schemes, this means that a special
      ansatz is used, and again it is sufficient to generate a reduced
      set of equations, see~\cite{auzingeretal15K1}.
\end{itemize}
Apart from these modifications, the basic algorithm remains unchanged.
\section{Modifications and extensions} \label{sec:ext}
\subsection{Splitting into more than two operators} \label{sec:OC-ABC}
Our algorithm directly generalizes to the case of
splitting into more than two operators.
Consider evolution equations where the right-hand side
splits into three parts,
\begin{equation} \label{ABC-problem}
\partial_t u(t) = F(u(t)) = A(u(t)) + B(u(t)) + C(u(t)),
\end{equation}
and associated splitting schemes,
\begin{subequations} \label{ABC-scheme}
\begin{equation} \label{ABC-scheme-1}
\nS(h,u) = \nS_s(h,\nS_{s-1}(h,\ldots,\nS_1(h,u))) \approx \phi_F(h,u),
\end{equation}
with
\begin{equation} \label{ABC-scheme-2}
\nS_j(h,v) = \phi_C(c_j\,h,\phi_B(b_j\,h,\phi_A(a_j\,h,v))),
\end{equation}
\end{subequations}
see~\cite{auzingeretal14a}.
Here the linear representation~\eqref{OC-Tay-1-AB} generalizes as follows,
with $A_j = a_j\,A,\, B_j=b_j\,B,\, C_j=c_j\,C $, and
$ {\bm k}=(k_1,\ldots,k_s) \in \NN_0^s $,\,
$ {\bm\ell} = (\ell_A,\ell_B,\ell_C) \in \NN_0^3 $:
\begin{equation} \label{OC-Tay-1-ABC}
\ddh{q}\,\nL(h)\,\Big|_{h=0}
= \sum_{|{\bm k}|=q} \dbinom{q}{{\bm k}}
  \prod\limits_{j=s \ldots 1}\;
  \sum_{|{\bm\ell}|=k_j} \dbinom{k_j}{{\bm\ell}}\,
                         C_j^{\ell_C}\,B_j^{\ell_B}\,A_j^{\ell_A}
  \;-\; (A+B+C)^q.
\end{equation}
On the basis of these identities, the algorithm from Section~\ref{subsec:ire}
generalizes in a straightforward way.
The Lyndon basis representing independent commutators now corresponds to Lyndon words
over the alphabet $ \{\mathtt{A,B,C}\} $, see~\cite{auzingeretal15K1}.
Concerning special cases (symmetries etc.) and parallelization,
similar considerations as before apply.
\subsection{Pairs of splitting schemes} \label{subsec:pairs}
For the purpose of adaptive time-splitting algorithms,
the construction of (optimized) pairs of schemes of orders
$ (p,p+1) $ is favorable.
Generating a respective set of order conditions can also be
accomplished by a modification of our code; the difference lies in
the fact that some coefficients are chosen a priori
(corresponding to a given method of order $ p+1 $),
but apart from this the generation of order conditions for an associated
scheme of order $ p $ works analogously as before.
Finding optimal schemes is then accomplished by tracing a large set
of possible solutions; see~\cite{auzingeretal15K1}.

\paragraph{Acknowledgements.}
This work was supported by the Austrian Science Fund (FWF) under grant P24157-N13,
and by the Vienna Science and Technology Fund (WWTF) under grant MA-14-002.

Computational results based on the ideas in this work have been achieved in part using the
Vienna Scientific Cluster (VSC).
%
%

\end{document}